\documentclass[12pt,a4paper,twoside]{amsart}
\usepackage{amsmath}
\usepackage{amssymb}
\usepackage{amsfonts}
\usepackage{a4}

\newtheorem{theorem}{Theorem}[section]

\newtheorem{proposition}[theorem]{Proposition}
\newtheorem{corollary}[theorem]{Corollary}
\theoremstyle{definition}
\newtheorem{definition}[theorem]{Definition}
\newtheorem{claim}[theorem]{Claim}

\newtheorem{remark}[theorem]{Remark}

\newcommand{\mA}{\mathbb A}

\newcommand{\mC}{{\mathbb C}}

\newcommand{\mE}{{\mathbb E}}
\newcommand{\mF}{\mathbb F}

\newcommand{\mV}{\mathbb V}

\newcommand{\mX}{\mathbb X}

\newcommand{\ho}{\hookrightarrow}
\newcommand{\Gg}{\gamma}

\newcommand{\bl}{\lambda}
\newcommand{\bL}{\Lambda}

\newcommand{\D}{\Delta}

\newcommand{\kk}{\kappa}

\newcommand{\ti}{\tilde}
\newcommand{\wt}{\widetilde}

\usepackage{color}
\usepackage{soul}

\newcommand{\sing}{\operatorname{sing}}
\newcommand{\ar}{ar}

\begin{document}

\title{On the  codimension of the singular locus}


\author{David Kazhdan}
\address{Einstein Institute of Mathematics,
Edmond J. Safra Campus, Givaat Ram 
The Hebrew University of Jerusalem,
Jerusalem, 91904, Israel}
\email{david.kazhdan@mail.huji.ac.il}

\author{Tamar Ziegler}
\address{Einstein Institute of Mathematics,
Edmond J. Safra Campus, Givaat Ram 
The Hebrew University of Jerusalem,
Jerusalem, 91904, Israel}
\email{tamarz@math.huji.ac.il}

\thanks{The second author is supported by ERC grant ErgComNum 682150. 
}

\begin{abstract} Let $k$ be a field and $V$ an $k$-vector space.
For a family $\bar P=\{ P_i\}_{1\leq i\leq c}, $ of polynomials on $V$, we denote by $\mX _{\bar P}\subset V$ the 
subscheme defined by the ideal generated by $ \bar P$. We show the existence of $\Gg (c,d)$ such that 
 varieties $\mX _{\bar P}$ are smooth outside of codimension $m$, if $\deg(P_i)\leq d$ and rank (strength)
$r_{nc}(\bar P)\geq \Gg (d,c) (1+m)^{\Gg (d,c)}$.
\end{abstract}

\maketitle

\section{Introduction} Let $k$ be a field and $\mV$ an $k$-vector space.
For a family $\bar P=\{ P_i\}_{1\leq i\leq c}, $ of polynomials on $V$ we denote by $\mX _{\bar P}\subset \mV$ the subscheme  defined by a the ideal 
$\{ P_i\}$ and by $\mX ^{\sing}_{\bar P}\subset \mX _{\bar P}$  the 
subscheme of singularities. We denote by $\kk (\bar P)$ the codimension of 
$\mX ^{\sing}_{\bar P}$ in $ \mX _{\bar P}$. If $c=1$ we write $\mX _P$ 
instead of $ \mX _{\bar P}$.

For a polynomial $P:\mV\to k$ we denote by $dP :\mV \to \mV^\vee$ the differential $dP (v)(h):= \partial P/\partial h(v),h\in \mV$. 
Then $\mX^{\sing}_P =\mX _P\cap dP ^{-1}(0)=\wt {dP}^{-1}(0)$ where $\wt {dP}(v):\mV \to \mA \oplus \mV ^\vee$ is given by  $\wt {dP}(v) = (P(v),dP (v))$.

For a polynomial $P:V\to k$  we denote by $\ti P :V^d\to k$ the multilinear symmetric form associated with $P$ defined by 
$\ti P(h_1, \ldots, h_d) : =  \Delta_{h_1} \ldots  \Delta_{h_d} P: V^d \to k$, where $\Delta_hP(x) = P(x+h)-P(x)$. 
The {\em rank} $r(P)$ is  the minimal number $r$ such that $P$ can be written in the form $P=\sum _{i=1}^rQ_iR_i$, where $Q_i,R_i$ are polynomials on $V$ of degrees $<d$.   We define the  {\em non-classical rank (nc-rank)} $r_{nc}(P)$ to be the rank of $\ti P$, and observe that if  char$(k) >d$ (or char$(k)=0$) then $r(P) \le r_{nc}(P)$,
since  $\ti P\circ \D =d! P$ where $\D :V\ho V^d$ is the diagonal imbedding.\\

The following statement  is proven in \cite{ah}.
\begin{claim}\label{serre} There exist a function $C(d,c,m)$ such that  $\kk (\bar P) \geq m$ for any family $\bar P=\{ P_i\}_{1\leq i\leq c}$  of polynomials $P_i$ on $V$, with  $\deg(P_i) \leq d$ for $1 \leq i\leq c$, such that $r(\bar P)\geq C(d,c,m)$ where $r(\bar P) $ is the algebraic rank (see section \ref{sec-def}).
\end{claim}

The goal of this paper is to obtain an explicit upper bound on $C(d,c,m)$. We first consider the special case when $c=1$.

\begin{theorem}\label{1} Let $d>0$. There exists a function $\Gg (d)$ such that for any field $k$, and a  polynomial $P$ of degree $d$ and of nc-rank $\geq \Gg (d) (1+m)^{\Gg (d)}$, we have 
$\kk (P)\geq m$.
\end{theorem}
\begin{remark} It is sufficient to consider the case when
 $P\in k[x_1,\dots ,x_n]$ is a homogeneous polynomial. Really for such a  polynomial $P$ of degree $d$ there exists a unique homogeneous polynomial $Q \in k[x_0,x_1,\dots ,x_n]$ of degree $d$ such that $Q(1, x_1,\dots ,x_n)=P(x_1,\dots ,x_n)$. Since $r(Q)\geq r(P)$ and $\kk (P)=\kk (Q)$ the inequality $\kk (Q)\geq m$ implies the inequality $\kk (P)\geq m$.
\end{remark}

Now we state the Theorem in  the general case. Let $\bar P=\{ P_i\}_{1\leq i\leq c}$ be a family of polynomials on $V$ of degrees $d_i$, $1\leq i\leq c$.

\begin{theorem}\label{main} There exists a function $\Gg (d,c)$ such that for any field $k$ 
and any family $\bar P$ of degree $\bar d$ and  nc-rank $\Gg (c,d) (1+m)^{\Gg (c,d)}$, we have $\kk (\bar P)\geq m$. \end{theorem}

\begin{remark} Our proof only shows that the dependence of $C$ on $c$ and $d$ is double exponential.  \end{remark}

Here is an outline of the proof. We denote by $d \bar P :V\to (\bL ^c( V))^\vee$ the  map such that 
$d\bar P (v){(h_1,\dots ,h_c)}:=\det (\D _{\bar P,h_1,\dots ,h_c}(v))$ where 
$\D _{\bar P, h_1,\dots ,h_c}(v) $ is the $(c\times c)$-matrix with elements 
$$\{  \partial P_i/\partial h_j(v)\}, \  1\leq i,j \leq c.$$ Let  $\wt {d\bar P}(v):= (\bar P(v),d\bar P (v))$. The following statement is well  known.

\begin{claim}\label{inverse} 
 $\mX^{\sing}_{\bar P} =\wt {d\bar P}^{-1}(0)$. 
\end{claim}
For any subspace  $L\subset \bL ^c(V)$ we denote by $p_L:\bL ^c(V)\to L^\vee$ the natural projection and by $\bar Q_L$ the composition 
$p_L\circ \wt {d\bar P} :V\to L^\vee$. 
We see that  for a proof of Theorem \ref{main} it is sufficient to find a {\it good}
upper bound on  $\dim (\wt {d\bar P}^{-1}(0))$. To obtain such a bound 
we construct a {\it large} subspace $L\subset \bL ^c(V)$ such that the map 
$\bar Q_L$ is of analytic rank $>\dim(L)$. Then (see  Claim \ref{eq} )
all fibers of $\bar Q_L$ are of dimension $\dim(V)-\dim(L)$. Since $\wt {d\bar P}^{-1}(0)) \subset \bar Q_L ^{-1}(0) $ we see 
that the codimension of $\mX ^{\sing}_{\bar P} $ in $\mX _{\bar P}$ is not less than $\dim(L)-c$.

\section{definitions}\label{sec-def}

Let $k$ be a field and $\bar k$ it algebraic closure.
 \begin{definition}Let $P$ be a polynomial of degree $d$ on a  $k$-vector space $V$.
\begin{enumerate} 
\item  We denote by $\ti P :V^d\to k$ the multilinear symmetric form associated with $P$ defined by 
$\ti P(h_1, \ldots, h_d) : =  \Delta_{h_1} \ldots  \Delta_{h_d} P: V^d \to k$, where $\Delta_hP(x) = P(x+h)-P(x)$.
\item For a family  $\bar P =\{P_i\} _{1\leq i\leq c}$, of polynomials on $\mV$ we denote by 
$\mX _{\bar V}\subset \mV$ the subscheme defined by the system $\{v: P_i(v)=0\}$ of equations.
\item For a polynomial $P \in k[x_1,\dots ,x_N] $ of degree $d>1$ we define the {\it  rank} $r(P)$ as the 
minimal number $r$ such that $P$ can be written
as a sum $P=\sum _{j=1}^rQ_iR_i$ where $Q_j,R_j \in k[x_1,\dots ,x_N] $ are polynomials of degrees $<d$ \footnote{ Also known 
as  {\em $h$-invariant}, or {\em strength}. 
For tensors it is known as the  {\em partition rank}.}. For $d=1$ we define the rank to be infinite and for $d=0$ we define it as $0$.
\item  We define the  {\em non-classical rank (nc-rank)} $r_{nc}(P)$ to be the rank of $\ti P$.
 \item For a $d$ tensor $T : V^d \to k$ we define the {\em partition rank (p-rank)}  $pr(T)$ as the minimal 
number $r$ such that $T$ can be written as a sum $T(x_1, \ldots, x_d)=\sum _{ J_i \subset [1,d], i \in [r]} Q_j((x_l)_{l \in J_i})R_j((x_l)_{l \in J^c_i})$ where $Q_j,R_j$ are degree $<d$ tensors.
\item If $k=\mF _q$ and  $P:V\to k$,  we define  
$b(P)= \sum _{v\in V}\psi (T( v)) /q^{\dim (V)}$ where $\psi :k\to \mC ^\star$ is a non-trivial additive
 character and define the {\it analytic rank } $ar(P)$ as  $-\log_q(b(\ti P))$.
\item For a linear subspace  $ L \subset k[x_1,\dots ,x_N]$ we define $r(L):=\min _{P\in  L -\{ 0\}}r(P)$, and $r_{nc}(L), ar(L)$ similarly.
\item For a linearly independent family  $\bar P =\{P_i\}_{1\leq i\leq c}$, of polynomials we define 
$r(\bar P):=r(L(\bar P))$, where  $L(\bar P) $ is the linear span of $\bar P$, and  $ r_{nc}(\bar P), ar(\bar P)$ similarly.
\item For a subscheme $\mX \subset \mV$ we define $r(\mX), r_{nc}(\mX), ar(\mX),$ as the maximal rank of $r(\bar P), r_{nc}(\bar P), ar(\bar P)$ for 
families   $\bar P =\{P_i\}_{1\leq i\leq c}$, such that $\mX =\mX _{\bar P}$.
\item For a family $\bar P$ we denote by $\kk (\bar P)$ the 
codimension of the singular locus $\mX ^{\sing}_{\bar P} \subset \mX _{\bar P}$.
\end{enumerate}
\end{definition}

\begin{remark}\label{pd} 
\begin{enumerate}
\item If char$(k) >d$ then $r(P) \le r_{nc}(P)$.
\item 
In low characteristic it can happen that $P$ is of high rank and $\ti P$ is of low rank, for example in characteristic $2$ the polynomial $P(x) = \sum_{1 <i<j<k <l\le n} x_ix_jx_kx_l$ is of rank $\sim n$, but of nc-rank $3$, see 
\item $ r_{nc}(P) \le  pr(\ti P) \le C_dr_{nc}(P)$, $C_d \le 4^d$. (see \cite{kz}). In char $>d$, $C_d=1$. 
\end{enumerate}
\end{remark}

\begin{remark}\label{11}\leavevmode
\begin{enumerate}
\item For any tensor $T$ we have that $b(T)\geq 0$ and does not depend on a choice of a character $\psi$.
\item For any tensor $T$, we have that $pr(T) \le -\log_q b(T)$ (see \cite{kz-approx}, \cite{lovett-rank}). 
\item $|b(P)|^{2^d} \leq b(\ti P) $.
\end{enumerate}
\end{remark}

\begin{claim}\label{ord}Let $\bar P=\{ P_i\}_{l=1}^c$ be 
 a family of polynomials with  $r(\bar P)\neq 0$. Then
\begin{enumerate} 
\item $\dim(L(\bar P))=c$.
\item There exists a basis $Q_i^j$ in $L(\bar P)$ such that 
\begin{enumerate} 
\item $\deg (Q^j_i)=j$.
\item For any $j$ the family $\bar Q^j=\{ Q_i^j\}$ is of rank $\geq r(\bar P)$.
\end{enumerate}
\end{enumerate}
\end{claim}
\begin{remark} Since the subscheme $\mX _{\bar P} $ depends only on the space $L(\bar P)$,  we can (and will)  assume that all the families $\bar P$ we consider satisfy the conditions of Claim \ref{ord} on $\bar Q$.
\end{remark}


\begin{definition}\label{def-rank}Let $L\subset k[x_1,\dots ,x_N]$ be a finite dimensional subspace.
\begin{enumerate}
\item We denote by $\phi _L:V\to L^\vee$ the map given by $\phi _L (v)(P):=P(v)$ for  $P\in L$ and  $v\in V$. 
\item For $\bl \in L^\vee$ we define $f(\bl):=|V| ^{-1}| \phi _L ^{-1}(\bl)|$.
\item We denote by $\hat f:L\to \mC$ the  Fourier transform of $f$ 
given by 
$$\hat f(l):=\sum _{\bl \in L^\vee}\psi (\langle \bl ,l \rangle)f(\bl).$$
\item For  a polynomial $P$  on $V, h\in V$ we write $P_h(v)=P(v+h)-P(v)$.
\item For a subspace $W\subset V$ we denote by $L(W)$ the subspace of polynomials on $V $ spanned by $P$ and $P_w, \ w\in W$. 
\end{enumerate}
\end{definition}

\section{Results from the additive combinatorics} 
From now on we fix a non-trivial additive character  $\psi :k\to \mC ^\star$.
 We denote $\mE_{s \in S} = |S|^{-1}\sum_{s \in S}$. Denote by $U_d$ the Gowers $U_d$-uniformity norm 
for functions $f:V \to \mC$ defined by:
\[
\|f\|_{U_d}^{2^d} = \mE_{x, v_1, \ldots. v_d \in V} \prod_{\omega \in \{0,1\|^d} f^{\omega} (x +\sum_{i=1}^d \omega_i v_i)
\]
where $f^{\omega} = \bar f$ if $|\omega| = \sum_{i=1}^d \omega_i$ is odd and otherwise  $f^{\omega} = f$.

\begin{claim}[\cite{kz-approx,lovett-rank}]\label{Ga}
For any $d>0$ and a polynomial $P:V \to k$  of degree $d$ we have $ar(P)\leq pr(\ti P)$, i.e.
\[
\mE_{h \in V^d} \psi(\ti P(h_1, \ldots, h_d))=\| \mE_{x \in V} \psi(P(x))\|^{2^d}_{U_d} > q^{-pr(\ti P)}.
\]
\end{claim}

 The following result is proven in \cite{Mi} (and independently, with a slightly weaker bound that depends on the field $k$, by \cite{janzer})
 
\begin{theorem}\label{ar}There exists a triply exponential function $C(d)$ and double exponential function $D(d)$ such that
$ pr(\ti P)\leq C(d)(ar(P)^{D(d)}+1)$ for all polynomials of degree $d$.
\end{theorem}

We rewrite Theorem \ref{ar} in the following way:
For any $d>0$ there exists $ A_d, B_d>0$ such that  for any 
polynomial $P:V \to k$  of degree $d$, we have that  if 
\[
\mE_{x \in V^d} \psi(\ti P(x))> q^{-s},
\]
then $pr(\ti P) <  A_d(s^{B_d}+1)$.  The dependence of the constants on $d$ is double exponential for $B_d$ snd triply exponential for $A_d$. 

Alternatively, for any  polynomial $P:V \to k$  of degree $d$  there exists $\alpha_d, \beta_d >0$ such that it  $pr(\ti P) >r$ we have that 
\[
|\mE_{x \in V} \psi(P(x))| \le \|\mE_{x \in V} \psi(P(x))\|_{U_d} < q^{-(\alpha_dr-1)^{\beta_d}}.
\]

\begin{claim}\label{eq} If $\ar (L)>\dim(L)$ then  $|f(\bl)/f(\mu)-1|\leq 1/2$ 
for all $\bl ,\mu \in L^\vee(W) $.
\end{claim}
\begin{proof} Since $f(\bl)= |L|^{-1}\sum _{l\in L}\psi (\langle\bl ,l\rangle)\hat f(l)$ and $\hat f(0)=1$
it is sufficient to show that $|\hat f(l)|\leq 1/q^{\dim( L)+1}$ for $l\in L-\{0\}$. But this inequality follows immediately from the definition of $\ar(L)$ and the inequality $\ar (L)>\dim(L)$.
\end{proof}

\section{The case of one a hypersurface}
 In this section we present a proof of Theorem \ref{main} in 
$c=1$.
\begin{theorem}\label{12} There exists a function $\Gg (d)$ such that $\kk (P)\geq m$ for any polynomial  $P$ degree $d$ and of nc-rank $\Gg (d) (1+m)^{\Gg (d)}$.
\end{theorem}

\begin{proof}

 To simplify notations we define {\em $$T=T(r,m,d):= [(r/A_d)^{1/B_d} -m]/2C_d ,$$}
 where $A_d, B_d$ are from Theorem \ref{ar}, and $C_d$ is from Remark \ref{pd}.

We start with the following statement.
\begin{proposition}\label{r} 
There exists a function $r(d)$ such that the following holds. For any finite field $k$ of characteristic $>d$, for any $m \geq 1$ and a polynomial $P:V\to k$ of degree $d$ and nc-rank $r_{nc}(P)>r(d)$ there exist $h_1,\dots ,h_m \in V$ such that the family
  $\bar Q:=\{P,P_{h_1} , \ldots, P_{h_m}\}$  is of nc-rank $ >
T$. 
\end{proposition}

\begin{proof}
Fix $d, m$. Since $P$ is of degree $d$ and $P_w$ is of degree $<d$, we need to find $\bar w = (w_1, \ldots , w_m)$ such  that $\{P_{w_1} , \ldots, P_{w_m}\}$ is of nc-rank $> T$. Note that for $r$ sufficiently large we have   $T\le r$. 
 
Suppose for any $\bar w$ we have that $\{P_{w_1} , \ldots, P_{w_m}\}$ is of nc-rank $\le T$. Then for  any $\bar w$ there exists $b \in k^m\setminus 0$ such that $\sum b_iP_{w_i}$ is of nc-rank $\le T$. By the pigeonhole principle there exists $b \in k^m\setminus 0$ such that  the nc-rank of $\sum b_iP_{w_i} \le T$ for a set $F$ of  $\bar w$ of size $\ge |V|^m/q^m$.  Fix this $b$.  W.l.o.g. $b_1 \ne 0$.  This implies that there exists  $(w_2, \ldots, w_m)$ such that 
$(w_1, w_2, \ldots, w_m) \in F$  for a set $F_1$ if $w_1$ of size  $\ge |V|/q^m$. 
Now subtract to get for  $q^{-m}|V|^2$ many pairs $w_1, w_1'$ we have  
\[
b_1(P_{w_1}-P_{w'_1})
\]
is of nc-rank $<2T$. So that $P_{w_1-w_1'}$ is of nc-rank  $<2T$ for  $q^{-m}|V|^2$ many $w_1, w_1'$.
But this implies by Remark \ref{pd} that the p-rank of $\ti P_{w_1-w_1'}$ is $<2TC_d$, so that
\[
 \|\psi(P)\|^{2^{d}}_{U_{d}} =\mE_w \|\psi(P_w)\|^{2^{d-1}}_{U_{d-1}} > q^{-2TC_d-m},
\]
so that $P$ is of nc-rank 
\[
<A_d (2TC_d+m)^{B_d}.
\] 
To get a contradiction we need
\[
A_d (2TC_d+m)^{B_d}<r.
\]
Namely
\[
T <[(r/A_d)^{1/B_d} -m]/2C_d.
\]
\end{proof}

\begin{corollary} $|\hat f(l)|\leq q^{-(\alpha_dT-1)^{\beta_d}}$ for $l\in L \setminus \{0\}$.
\end{corollary}

Now we prove Theorem \ref{12}.\\

We write  $r=m^a , \ a>0$. Then Proposition \ref{r} guaranties the existence 
of $h_1,\dots ,h_m\in V$ such that the family
  $\bar Q:=\{P,P_{h_1} , \ldots, P_{h_m}\}$  is of nc-rank $ >[(r/A_d)^{1/B_d} -m]/2C_d$. Assuming that $a(B_d)^{-1} \geq 2$ we see that for $m>m_0(d)$ we have that $r_{nc}(\bar Q)\geq c(d)m^{a(B_d )^{-1} }$. As follows from   Proposition \ref{ar} we have 
$\ar(\bar Q)\geq c'(d)(m^{aB^{-1}_d} -1)^{\beta_d}$ for some  $c'(d)>0 $. Choose $a$ such that $aB^{-1}_d \beta_d  \geq 2 $. Then for $m>m_1(d)$  we have $\ar(\bar Q) \ge m+1$
  It is easy to see the existence of  $\Gg (d)>0$ such  that $T(m)>m$ for 
$m=r^{\Gg (d)}$. 
  
It follows now from Claim \ref{eq} that 
$|\mX _{\bar Q}(k_s)| \ll q^{s(\dim (V)-r^{\Gg (d)})}, s\geq 1$ where $k_s/k$ is the extension of degree $s$. Therefore (see 
 Section $3.5$ of \cite{kz})  $\dim (\mX _{\bar Q})=\dim(V)-r^{\Gg (d)} $. The arguments of Section $3.6$ of \cite{kz} show the the same equality is true in the case of an arbitrary field. 
 Since (by Claim \ref{inverse}) $\mX _P^{\sing}\subset \mX _{\bar Q}$ we see that $\kk (P)\geq m$.
\end{proof}

\section{higher codimension}
 Fix numbers $s$ and $n$ and denote by $m = \lfloor n/s \rfloor$, and by  $l:=n-sm$
Let $L=k^s, V= L\otimes k^n= L_1\times \dots L_n, L_i=L , W_j=V_{sj} \times \dots \times V_{s(j+1)-1}$ for $1\leq j\leq m$. So $V=W_1\times \dots \times W_m\times k^l$.

We denote by $y_i^j$ for  $1\leq i\leq n$, $1\leq j\leq s$ the natural coordinates on $V$ and
define a polynomial  of degree $s$,  $F^n _s :V\to k$ by 

$$
 F^n_s(y^1,  \dots, y^s):= \sum_{1 \le i_1<\ldots < i_s\le n} \det(y^1_{i_1}, \ldots, y^s_{i_s})
$$

\begin{claim}There exists $c_0 >0$ such that $b( F^n_s) \le q^{-c_0n/s}$.
\end{claim}
\begin{proof} Consider first the case $s=2$. Then 
$$P (a_1,\dots ,a_n ; b_1,\dots ,b_n) =\sum_{1\leq i<j\leq n}(a_ib_j-a_jb_i).$$

Let $Y=\{ \bar b| P(\bar a, \bar b)=0, \ \forall \bar a \in k^n\}$. Performing the summation over $\bar a$ we see that 
$b(P)=q^{-2n}|Y|$.

Let  $$Q_i:=\partial P/\partial a_i= \sum _{j>i}b_j-\sum _{j<i}b_j, \quad 1\leq i\leq n-1. $$
 Define   $Y\subset W=k^n$ is defined by the system $Q_i(w)=0$ of 
equations. 
Since 
$$Q_{i+1} (a_1,\dots ,a_n ; b_1,\dots ,b_n)-Q_i (a_1,\dots ,a_n ; b_1,\dots ,b_n) =b_{i+1}-b_i,$$
 we see the $Y$ is defined by the system $b_{i+1}-b_i=0$ of 
equations. So $b(F^n_2)=q^{1-n}$.

\bigskip

Consider now the general case. We write $y^t=(a^t_1,\dots ,a^t_n )$. Let 
$Q_i:= \partial P/\partial y^1_i$ for $1\leq i\leq n$. We define $Y_s\subset W:= k^{n(d-1)}$ as the set of solutions of the system $\{Q^1_i(w)=0\}_{1\leq i\leq n-1}$.

As before $ar(F^n_s)=q^{-sn}|Y_s|$. 

Let $Z _s\subset k^{(s-1)n}$ be defined by the system $\{ R_j=0\}_{1\leq j\leq n/s} $ of equations where $R_j:= Q^1_{s(j+1)} -Q^1_{sj}$, $1\leq j\leq m$. Since $Y_s \subset Z_s$ it is sufficient to show that $|Z_s|\leq q^{(ns-m)}$.

By the construction $R_j$ are non-zero polynomials on $W_j$, $1\leq j\leq m$. Let 
$Z^j_s\subset W_j$ be defined by the equation $R_j(w)=0$. Since $R_j$ is irreducible it follows from the Weil's bounds that $|Z^j_s|\leq (1+s^2q^{-1/2})q^{s^2-1}$. Since 
$Z_s=Z_s^1\times \dots \times Z_s^m\times k^e$ we see that $|Z_s|\leq q^{(ns-m)}$
\end{proof}


Denote $d_i'=d_i-1$, and let $e=\sum d'_i$, and. denote $f_i = \binom{e}{d_i'}$.  Denote $S(d;e)$ subsets of $[e]=\{1, \ldots, e\}$ of size $d$.  
  Denote by $X_i$ the set of subsets of $[1,e]$ of size $d'_i$. We write $V_i$ for the spaces  of $ k[X_i]$ of $k$-valued functions on $X_i$. These spaces have bases parametrized by elements of $X_i$ and 
define $W_i:=V_i\otimes R, R:=k^n$ and $W:= \oplus _iW_i$.  So a point $w\in W$ has coordinates $w_m^{i,x_i}$ , $1\leq m\leq n$, $1 \leq i\leq s ,x_i\in X_i$. For any 
sequence $\bl =(i_1<\dots <i_s)$  and $\bar x=(x_1, \dots ,x_s)$, $x_j\in X_j$ we denote by 
$h_{\bl ,\bar x}$ the polynomial on $W$ which is the determinant of an $s\times s$ matrix with elements $a_{ij}=w_{l_i}^{j,x_j}$.

Consider the function 
\[ \begin{aligned}
&G^n_s((y^{1,x_1})_{x_1 \in X_1}, \ldots   (y^{s,x_s})_{x_s \in X_s})
= \sum_{x_i \in X_i, \bigcup_{i=1}^s x_i = [e]}  \sum_{1 \le i_1<\ldots < i_s\le n} \det(y^{1,x_1}_{i_1}, \ldots, y^{s, x_s}_{i_s})
\end{aligned}\]
Then $G^n_s$ is a sum of the polynomials  $h_{\bl ,\bar x}$. 
Now the restriction of  $G^n_s $ onto $ k^s\otimes R $  that is the polynomial
$G^n_s \circ \phi : k^s\otimes R \to k $ coincides with the polynomial $F^n_s$, thus we obtain:

\begin{claim}\label{ugly}
 $b(G^n_s) \le q^{- c_0n/s}.$
\end{claim}

Consider the collection of polynomials $ F^n_{i,s} : k^{snm} \to k$ defined by
\[
F^n_{i,s}(y_1^1, \ldots,  y_m^1,  \ldots, y_1^s, \ldots, y_m^s) : =  F_s^n( y_i^1,\ldots,  y_i^s); \qquad 1 \le i \le m.
\]
Then the bias of any non trivial  linear combination of this collection is $ \le q^{-c_0n/s}$. Similarly $\{G^n_{i,s}\}_{i=1}^m$.


\begin{theorem}\label{m}  For any $m,d, s >0$, $1 <  d_i \le d, 1 \le i \le s$  there exists $C=C(m, d, s)>0$ such for  any polynomial family $\{P^1, \ldots, P^s\}$ on a $k$ vector space $V$,  $P^i$ of  degrees $d_i$ for $1 \le i \le s$,  and nc-rank $>Cr^C$ the following holds:
 there exist $t= t(r,s,m)$ and $w_1, \ldots, w_{tm} \in V$,  such that the collection  $\bar Q$ of functions 
 $$\{P^1, \ldots, P^s, \ F^t_{i,s}(P^1_{w_1} , \ldots, P^1_{w_{tm}},  \ldots,P^s_{w_1} , \ldots, P^s_{w_{tm}}); \ i \in [m]\}$$
 is of nc-rank $>r$.\end{theorem}
 
 {\color{blue}  I will assume that $d_i >1$ this is ok ? }
 \begin{proof}
 

  To simplify notations we write  $T=T(m,s,d):=m^{-1}a_1(s,d)r^{a_2(s,d)}-m$, where the constants $a_1(s,d), a_2(s,d)$ will be chosen later and will depend on the constants  $A_d, B_d, \alpha_d, \beta_d$ from Proposition \ref{ar}.

\begin{proposition}\label{rh}  For any  polynomials $P^1, \ldots, P^s:V\to k$ of degrees $d_i, 1 < i \le s$,  and nc-rank $r>r(d,s)$ there exist $t>0$ and  $w_1,\dots ,w_{tm} \in V$, such that the family
 $$\bar Q:=\{P^1, \ldots, P^s, \ F^t_{i,s}(P^1_{w_1} , \ldots, P^1_{w_{tm}},  \ldots,P^s_{w_1} , \ldots, P^s_{w_{tm}}); \ i \in [m]\}$$
is of nc-rank $ >T$.

\end{proposition}

\begin{proof} 
Fix $d, m$, $s>1$. 
We need to find $\bar w = (w_1, \ldots , w_{tm})$ such that
\[
\{P^1, \ldots, P^s, \ F^t_{i,s}(P^1_{w_1} , \ldots, P^1_{w_{tm}},  \ldots,P^s_{w_1} , \ldots, P^s_{w_{tm}}), \ i \in  [m]\}
\]
 is of nc-rank $> T$. 
 
Suppose for any  $\bar w=(w_1, \ldots, w_{tm})$, the system 
\[
(*) \quad \{P^1, \ldots, P^s, F^t_{i,s}(P^1_{w_1} , \ldots, P^1_{w_{tm}},  \ldots,P^s_{w_1} , \ldots, P^s_{w_{tm}}), \quad i \in  [m]\}
\]
of polynomials is of nc-rank $<T$, but the system $\{P^1, \ldots, P^s\}$ is of nc-rank $>T$. 

Then for  any $\bar w$ there exists $b  \in k^s, c \in k^{m}\setminus 0$ such that 
\[
\sum_{j=1}^s b_jP^j+\sum_{i=1}^m c_iF^t_{i,s} (P^1_{w_1} , \ldots, P^1_{w_{tm}},  \ldots,P^s_{w_1} , \ldots, P^s_{w_{tm}})
\]
 is of nc-rank $<T$. By the pigeonhole principle there exists $b  \in k^s, c \in k^{m}\setminus 0$ such that  $\sum_{j=1}^s b_jP^j+\sum_{i=1}^m  c_iF^t_{i,s} $ is of nc-rank $<T$  for a set of 
$\bar w$ of size $\ge |V|^{tm}/q^{m+s}$.  Fix these $b,c$. 
Since the degree of $\sum_{i=1}^m c_i  F^t_{i,s}$ is $e=\sum_{i=1}^sd'_i > \max d_i$, we get that
the p-rank of 
\[
\Delta_{u_e} \ldots \Delta_{u_1}\sum_{i=1}^m c_i  F^t_{i,s} ( P^1_{w_1}(x) , \ldots,  P^1_{w_{tm}}(x),  \ldots,P^s_{w_1}(x) , \ldots,  P^s_{w_{tm}}(x)) < TC_e,
\]
where $C_e$ is the constant from Remark \ref{pd}. Thus  
\[
 \psi \left( \Delta_{u_e} \ldots \Delta_{u_1}\sum_{i=1}^m c_i  F^t_{i,s} ( P^1_{w_1}(x) , \ldots,  P^1_{w_{tm}}(x),  \ldots,P^s_{w_1}(x) , \ldots,  P^s_{w_{tm}}(x)) \right)  
 > q^{-(TC_e+m+s)}.
\]
for some non trivial character $\psi:k \to S^1$. 

The argument of $\psi$ is 
\[ \begin{aligned}
&G(u_1, \ldots, u_e) =\sum_{i=1}^m c_i  G^t_{i,s} ( (\ti P^{1, x^1_1}_{w_1}), \ldots,  (\ti P^{1, x^1_{tm}}_{w_{tm}}),  \ldots ,(\ti P^{s, x^s_1}_{w_1}) , \ldots,  (\ti P^{s, x^s_{tm}}_{w_{tm}})) 
\end{aligned}\]
where $(\ti P^{i, x^i_j}_{w_j})$ denote the tuple $(\ti P^{i, x^i_j}_{w_j})_{x^i_j \in X_i}$, and   $\ti P^{i, x^i_j}_{w}(u) = \ti P^i_w(u_{i_1},\ldots, u_{i_{d_i'}})$ if $x_j^i= \{i_1,\ldots, i_{d'_i}\}$.

By Claim \ref{ugly} the function  
\[
g(y)=\sum_{i=1}^m c_i  G^t_{i,s}((y^{1, x^1_1}_{1}), \ldots (y^{1, x^1_{tm}}_{tm}), \ldots, (y^{s, x^s_1}_{1}),, \ldots, ,(y^{s, x^s_{tm}}_{tm})),
\]
where $(y^{i, x^i_j}_{j})$ denote the tuple $(y^{i, x^i_j}_{j})_{x^i_j \in X_i}$, is of p-rank $\ge c_0t/2sc$.  Denote $f_i = \binom{e}{d'_i}$, and let $f=\sum_{i=1}^s f_i$

 Consider the function $\psi \circ g: k^{tmf} \to \mathbb C$. By Fourier analysis,
\[
 \psi \circ g(y) =  \sum_{\chi \in \hat{k}^{tmf}}a_{\chi} \chi(y)
\]  
Since $g$ is multilinear, and by Claim \ref{ugly} we have that the coefficient of the trivial character is bounded :
\[
|\mE_{y}  \psi \circ g(y) | < q^{- c_0t/s}
\]
We get
\[
(\star)  \quad  \left | \mE_u \sum_{ \chi \ne 1, \chi \in \hat k^{tmf }}a_{\chi }\chi(  (\ti P^{1, x^1_1}_{w_1}), \ldots,  (\ti P^{1, x^1_{tm}}_{w_{tm}}),  \ldots ,(\ti P^{s, x^s_1}_{w_1}) , \ldots,  (\ti P^{s, x^s_{tm}}_{w_{tm}})) \right |
\]
is greater equal than
\[
 \left | \mE_u \sum_{  \chi \in \hat k^{tmf }}a_{\chi }\chi((\ti P^{1, x^1_1}_{w_1}), \ldots,  (\ti P^{1, x^1_{tm}}_{w_{tm}}),  \ldots ,(\ti P^{s, x^s_1}_{w_1}) , \ldots,  (\ti P^{s, x^s_{tm}}_{w_{tm}})) \right |- q^{- c_0t/s}
\]
so that $(\star)$  is   $\ge q^{-TC_e-m-s} -  q^{- c_0t/s }$.

Choose $t$ so that $q^{-  c_0t/s} < q^{-2(TC_e+m+s)}$, i.e. $$t= \lceil ((s/c_0) 2(TC_e+m+s))\rceil.$$
Thus for any $\bar w$ we can find $\chi \ne 1, \chi \in \hat k^{tmf } $ such that 
\[
  \left | \mE_u a_{\chi }\chi(  (\ti P^{1, x^1_1}_{w_1}), \ldots,  (\ti P^{1, x^1_{tm}}_{w_{tm}}),  \ldots ,(\ti P^{s, x^s_1}_{w_1}) , \ldots,  (\ti P^{s, x^s_{tm}}_{w_{tm}})) \right | \ge q^{-2(TC_e+m+s)} .
\]
Set $D=2(TC_e+m+s)$.
By the pigeonhole principle there is a set $A$ of $\bar w$ of size $\ge q^{-tmf}|V|^{tm}$, and $\chi \ne 1, \chi \in \hat k^{tmf } $ such that for all $\bar w \in A$  we have 
\[
  \left |\mE_u a_{\chi }\chi(  (\ti P^{1, x^1_1}_{w_1}), \ldots,  (\ti P^{1, x^1_{tm}}_{w_{tm}}),  \ldots ,(\ti P^{s, x^s_1}_{w_1}) , \ldots,  (\ti P^{s, x^s_{tm}}_{w_{tm}})) \right |\ge q^{-D}
\]
This implies that   there is a non trivial linear combination of 
\[
 (\ti P^{1, x^1_1}_{w_1}), \ldots,  (\ti P^{1, x^1_{tm}}_{w_{tm}}),  \ldots ,(\ti P^{s, x^s_1}_{w_1}) , \ldots,  (\ti P^{s, x^s_{tm}}_{w_{tm}})
\]
for all $\bar w \in A$, which is of p-rank $\le A_{d}(D^{B_d}+1)$.

Without loss of generality the coefficient of
the first term is not zero. 
This implies that  for a set of $w_1, w_1'$ of size  $\ge q^{-tmf}|V|^2$, we have that 
\[
 \ti P^{1, x^1_1}_{w_1}(u)-\ti P^{1, x^1_1}_{w'_1}(u)
\]
is of p-rank $\le 2A_{d}(D^{B_d}+1)$

But this implies that for a set of size  $\ge q^{-tmf}|V|$ of $w$ we have $ \ti P^{1, x^1_1}_{w}$
is of p-rank $\le 2A_{d}(D^{B_d}+1)$

So that 
\[
\mE_{w \in V}  \mE_{u \in V^{d'_1}} \psi( \ti P^{1}_{w}(u)) \ge q^{-2A_{d}(D^{B_d}+1) -tmf} .
\]
But this now implies that
\[
  \|\psi(P^1)\|^{2^{d_1}}_{U_{d_1}}\ge  q^{-2A_{d}(D^{B_d}+1) -tmf},
\]
and thus $P$ is of $p$-rank  
\[
 \le A_d(( 2A_{d}(D^{B_d}+1) +tmf)^{B_d}+1)
\]
To get a contradiction we need the above to be $<r$.
Namely, 
\[
 A_d(( 2A_{d}((2(TC_e+m+s))^{B_d}+1) +tmf)^{B_d}+1)<r
\]
The above condition can be replaced by a condition of the form:
\[
TC_e< m^{-1}a_1(s,d)r^{a_2(s,d)}-m.
\]
with $a_1(s,d), a_2(s,d)$ depending polynomially on $s$ and on $A_d, B_d,$ from Proposition \ref{ar}.
\end{proof}

We can consider the collection $\bar Q$ of \ref{m} as a map 
$\bar Q:\mA ^N\to \mA ^f$. Since (by Claim \ref{inverse})  $\mX ^{\sing}_{\bar P}\subset \bar Q^{-1}(0)$ we derive  Theorem \ref{main} from Proposition \ref{rh} using the same arguments as in the derivation of Theorem \ref{1} from Proposition \ref{r}.
\end{proof}

\end{document}